\documentclass[letterpaper, 10 pt, conference]{ieeeconf}

\IEEEoverridecommandlockouts
\usepackage{amsmath, amssymb, amsfonts}
\usepackage{mathtools}
\usepackage{color}
\usepackage{graphicx}

\DeclareMathOperator{\median}{med}
\DeclareMathOperator*{\esssup}{ess\,sup}
\DeclareMathOperator*{\esslimsup}{ess\,lim\,sup}

\newcommand{\mmax}{_\mathrm{max}}

\newcommand{\ac}{_\mathrm{ac}}
\newcommand{\dc}{_\mathrm{dc}}
\newcommand{\const}{\mathrm{const}}
\newcommand{\trans}{^\mathsf{T}}
\newcommand{\Linf}{L^\infty}

\newcommand{\Gux}{\Gamma^{u \mapsto x}}
\newcommand{\Gfx}{\Gamma^{f \mapsto x}}
\newcommand{\Guy}{\Gamma^{u \mapsto y}}
\newcommand{\Gfy}{\Gamma^{f \mapsto y}}

\newcommand{\gux}{\gamma^{u \mapsto x}}
\newcommand{\gfx}{\gamma^{f \mapsto x}}
\newcommand{\guy}{\gamma^{u \mapsto y}}
\newcommand{\gfy}{\gamma^{f \mapsto y}}

\newcommand{\dif}{\mathrm{d}}
\newcommand{\ee}{\mathrm{e}}

\newcommand{\pb}[1]{\big( #1 \big)}

\newcommand{\px}[1]{\left( #1 \right)}

\newcommand{\abs}[1]{\vert #1 \vert}

\newcommand{\set}[1]{\lbrace #1 \rbrace}

\newcommand{\norm}[1]{\Vert #1 \Vert}
\newcommand{\normb}[1]{\big\Vert #1 \big\Vert}
\newcommand{\normB}[1]{\Big\Vert #1 \Big\Vert}
\newcommand{\normx}[1]{\left\Vert #1 \right\Vert}

\newcommand{\mean}[1]{\langle #1 \rangle}
\newcommand{\meanb}[1]{\big\langle #1 \big\rangle}

\newcommand{\bracksx}[1]{\left\lbrack #1 \right\rbrack}

\newcommand{\pd}[2][]{\frac{\partial #1}{\partial #2}}

\newcommand{\bbC}{\mathbb{C}}
\newcommand{\bbR}{\mathbb{R}}
\newcommand{\bbN}{\mathbb{N}}

\newcommand{\calU}{\mathcal{U}}

\newcommand{\calX}{\mathcal{X}}

\newcommand{\calS}{\mathcal{S}}

\newcommand{\calE}{\mathcal{E}}

\newtheorem{remark}{Remark}
\newtheorem{definition}{Definition}
\newtheorem{assumption}{Assumption}
\newtheorem{proposition}{Proposition}
\newtheorem{lemma}{Lemma}
\newtheorem{theorem}{Theorem}
\newtheorem{corollary}{Corollary}

\title{\LARGE \bf Period-Aware Asymptotic Gain\\with Application to a Periodically Forced Synchronization Circuit}

\author{Anton Ponomarev$^{1}$ and Lutz Gröll$^{1}$ and Veit Hagenmeyer$^{1}$
\thanks{*The authors gratefully acknowledge funding by the German Federal Ministry of Research, Technology and Space (BMFTR) within the Kopernikus Project ENSURE ‘New ENergy grid StructURes for the German Energiewende’ (03SFK1B0-3).}
\thanks{$^{1}$The authors are with the Institute for Automation and Applied Informatics, Karlsruhe Institute of Technology, Eggenstein-Leopoldshafen, 76344 Baden-Württemberg, Germany {\tt\small \{anton.ponomarev, lutz.groell, veit.hagenmeyer\}@kit.edu}}%
}

\begin{document}

\maketitle
\thispagestyle{empty}
\pagestyle{empty}

\begin{abstract}
The classical asymptotic gain (AG) is a concept known from the input-to-state stability theory. Given a uniform input bound, AG estimates the asymptotic bound of the output. Sometimes, however, more information is known about the input than just a bound. In this paper we consider the case of a periodic input. Under the assumption that the system converges to a periodic solution, we introduce a new gain, called period-aware asymptotic gain (PAG), which employs periodicity to enable a sharper asymptotic estimation of the output. Since the PAG can distinguish between short-period (``high-frequency'') and long-period (``low-frequency'') signals, it is able to rigorously quantify such properties as bandwidth, resonant behavior, and high-frequency damping. We discuss how the PAG can be computed and illustrate it with a numerical example from the field of power electronics.
\end{abstract}

\section{Introduction}

\emph{Asymptotic gain} (AG) is a convenient characteristic of nonlinear control systems. It relates the uniform input bound and the asymptotic bound of the output. AG is a part of the \emph{input-to-state stability} (ISS) theory~\cite{sontagInputtoStateStabilityProperty1995}.

One application of the AG is in the \emph{small-gain} stability analysis of dynamical networks~\cite{dashkovskiyISSSmallGain2007}. Based on the way each subsystem damps or amplifies the input, conclusions can be made about whether the signals eventually vanish while looping through the network. On a very high level, the small-gain idea can be compared to the \emph{passivity} and \emph{dissipativity} conditions which operate with flows of energy rather than signal bounds~\cite{vidyasagarNewPassivitytypeCriteria1979}. 

If the input is affected by a bounded persistent perturbation, both AG and energy-based approaches can bound the output. An important distinction is that AG deals with uniform (supremum) bounds whereas energy is a more delicate measure of an ``effective'' magnitude. In practice, input perturbation often \emph{oscillates} inside the supremum bounds which makes the AG approach very conservative. In the present contribution we tweak the AG concept to try and capture some aspects of averaging and oscillatory behavior, at the same time preserving the magnitude bounds that are characteristic of ISS and not of the energy-based tools.

Consider a linear system and its Bode diagram, particularly its magnitude part. At face value, it shows the gain of pure sine signals as they go through the system. To carry this notion over to nonlinear systems, interaction between different frequencies has to be taken into account. This gives rise to such methods as \emph{harmonic balance}, \emph{Volterra series}, \emph{nonlinear frequency response}~\cite{pavlovFrequencyResponseFunctions2007}, etc. One may also read from the Bode diagram that the system blocks inputs in a certain frequency band, or that there is resonance at a certain frequency. Our variation of AG indicates similar properties of a nonlinear system forced by a \emph{periodic input}. Note the difference: whereas methods like harmonic balance assume \emph{harmonic} inputs, we consider \emph{arbitrary} periodic signals.

The proposed AG is called a \emph{period-aware asymptotic gain} (PAG) and applies to the case of periodic inputs. Such inputs or disturbances do appear in some applications~-- e.g., in Section~\ref{se: example} we consider an example from power electronics where disturbance is often dominated by the harmonics of the main grid frequency. The idea is to treat the input as the sum of a constant signal (\emph{DC component}) and a periodic signal with zero average (\emph{AC component}). Knowing the magnitude bounds of each component, we can estimate the asymptotic bounds of the DC and AC components of the output~-- this essentially defines the PAG.

PAG must be compared to the recently introduced ISS variations that bound the output magnitude by the moving average of the input~\cite{efimovISSRespectAverage2024,haimovichInputpowertostateStabilityTimevarying2025a}. PAG is different in two ways:
\begin{itemize}
    \item PAG depends on the \emph{magnitudes} of the AC/DC components of the input, not on the moving average which would be an energy-like approach;
    \item PAG is a \emph{vectorial} function~-- it maps two arguments (the bounds on the AC/DC components of the input) to two bounds on the AC/DC components of the output.
\end{itemize}
However, PAG does not replace ISS: it is invoked \emph{after} convergence to a periodic solution is established~-- for instance, by means of ISS, see Section~\ref{se: scenario}.

If, looking at the PAG plot, one sees that the gain of short-period inputs is far below 1, then it can be interpreted similarly to the high-frequency attenuation property of linear systems. It is readily assumed and widely used in the engineering practice that such frequency-dependent properties transfer from linearization to nonlinear systems under the small-signal approximation~\cite{wangHarmonicStabilityPower2019}. The PAG analysis may help quantify this linearization-based point of view.

In Section~\ref{se: main} we define the PAG and outline the framework for its use. The formula for PAG in the linear case is given in Section~\ref{se: linear}. For nonlinear systems, we derive a linearization-based approximation in Section~\ref{se: nonlinear}. The estimations are tried out in Section~\ref{se: example} in a nonlinear example. The results show that the PAG bounds are tighter than the classical period-agnostic AG.

\section{Preliminaries}
\label{se: preliminaries}

Let us introduce the notation and some basic notions. Given $x, y\in\bbR^n$ and $M \in \bbR^{m\times n}$:
\begin{itemize}
    \item $x\trans$ is the transpose of $x$;
    \item $x \cdot y = x\trans y$;
    \item $\norm{x} = \sqrt{x\cdot x}$;
    \item $\norm{M} = \max_{\norm{x} = 1} \norm{Mx}$;
    \item $\norm{x}_1 = \sum_{k=1}^n \abs{x_k}$;
    \item $x \preceq y \iff x_k \leq y_k$ for all $k = 1, 2, \dots, n$.
\end{itemize}
Given $f \in \Linf(\bbR, \bbR^n)$:
\begin{itemize}
    \item $\norm{f}_\infty = \esssup_{t \in \bbR} \norm{f(t)}$;
    \item given $T>0$, the \emph{moving $T$-average} of $f$ is
    \begin{equation}
        \mean{f}_T (t) := \frac1T \int_{t-T}^t f(\tau) \,\dif\tau;
    \end{equation}
    \item we shall use the $\bbR^2$-vector
    \begin{equation}
        \rho_T(f) = \begin{bmatrix*}[l]
            \esslimsup_{t\to\infty} \normb{\mean{f}_T(t)} \\
            \esslimsup_{t\to\infty} \normb{f(t) - \mean{f}_T(t)}
        \end{bmatrix*}
    \end{equation}
    as a measure of the \emph{asymptotic AC/DC magnitude} of $f$ with respect to the window length $T>0$;
    \item if $f$ is $T$-periodic then its \emph{AC/DC decomposition} is
    \begin{equation}
        f(t) = f\dc + f\ac(t)
    \end{equation}
    where $f\dc = \mean{f}_T = \const$ and $f\ac$ is $T$-periodic with $\mean{f\ac}_T = 0$;
    \item if $f$ is $T$-periodic then
    \begin{equation}
        \rho_T(f) = \begin{bmatrix*}[l]
            \norm{f\dc} \\
            \norm{f\ac}_\infty
        \end{bmatrix*}.
    \end{equation}
\end{itemize}
Given $f \in \Linf([0,T], \bbR^n)$:
\begin{itemize}
    \item the \emph{geometric median} of $f$ is
    \begin{equation}
        \label{eq: median definition explicit}
        \median f := \arg\min_{\mu \in \bbR^n} \frac1T \int_0^T \norm{f(t) - \mu} \,\dif t;
    \end{equation}
    \item an equivalent \emph{implicit} definition of $\median f$ is \cite[Eq.~(5)]{cardotEfficientFastEstimation2013}
    \begin{equation}
        \label{eq: median definition implicit}
        \frac1T \int_0^T \frac{f(t) - \median f}{\norm{f(t) - \median f}} \,\dif t = 0
    \end{equation}
    where the 0/0 irregularity is resolved to 0;
    \item the \emph{mean absolute deviation} of $f$ from its median is
    \begin{equation}
        D_{\median} f :=  \frac1T \int_0^T \norm{f(t) - \median f} \,\dif t,
    \end{equation}
    i.e., the result of the minimization in~\eqref{eq: median definition explicit}.
\end{itemize}

\section{Period-Aware Asymptotic Gain}
\label{se: main}

\subsection{Definition}

Consider the system
\begin{subequations}
    \label{eq: system general}
    \begin{align}
        \dot x &= Ax + Bu + Ff(x,u), \\
        y &= Cx + g(x)
    \end{align}
    where $x \in \bbR^n$, $u \in \bbR^m$, $y \in \bbR^p$, matrix $A$ is Hurwitz, and functions $f$ and $g$ are globally smooth.
\end{subequations}

It is often necessary to limit the admissible initial states and inputs. To this end, we introduce:
\begin{itemize}
    \item $\calX_0 \subset \bbR^n$~-- an open set of initial states with $0 \in \calX_0$;
    \item $\calU = \set{u \in \Linf(\bbR, \bbR^m) \colon \norm{u}_\infty < u\mmax}$~-- a set of inputs bounded by some $u\mmax \in (0, \infty]$.
\end{itemize}

Recall the following classical definition.

\begin{definition}
    \label{def: gain}
    A non-decreasing $\gamma\colon [0, u\mmax) \to [0, \infty)$ is called a (conservative) \emph{asymptotic gain} (AG) of~\eqref{eq: system general} if for every $u \in \calU$ and $x(0) \in \calX_0$
    \begin{equation}
        \limsup_{t\to\infty} \norm{y(t)} \leq \gamma(\norm{u}_\infty).
    \end{equation}
    The minorant of conservative AGs is called the \emph{exact} AG.
\end{definition}

Given an input magnitude bound, AG provides an asymptotic bound of the output magnitude. If, in addition to the magnitude bound, \emph{periodicity} of the input is assumed, then it is natural to expect that the output bound can be tightened. This tighter bound is provided by the \emph{period-aware} AG defined as follows.

In the notation of Section~\ref{se: preliminaries}, let
\begin{multline}
    \calU_T = \set{u \in \Linf(\bbR, \bbR^m) \colon u \text{\ is\ $T$-periodic,} \\
    \norm{\rho_T(u)}_1 < u\mmax}.
\end{multline}
Note that $\calU_T \subset \calU$ because for a $T$-periodic $u$
\begin{equation}
    \label{eq: 1-norm inf-norm inequality u}
    \norm{u}_\infty \leq \norm{u\dc} + \norm{u\ac}_\infty \equiv \norm{\rho_T(u)}_1.
\end{equation}

\begin{definition}
    \label{def: pag}
    With the notation of Section~\ref{se: preliminaries}, a $\preceq$-non-decreasing function $\gamma_T\colon [0, u\mmax)^2 \to [0, \infty)^2$ is called a (conservative) \emph{period-aware asymptotic gain} (PAG) of~\eqref{eq: system general} if for every $u \in \calU_T$ and $x(0) \in \calX_0$
    \begin{equation}
        \rho_T(y) \preceq \gamma_T\pb{\rho_T(u)}.
    \end{equation}
    The $\preceq$-minorant of conservative PAGs is the \emph{exact} PAG.
\end{definition}

The following claim is obvious since $\calU_T \subseteq \calU_{kT}$ for $k \in \bbN$.

\begin{proposition}
    \label{pr: integer periods}
    For all $k \in \bbN$ and $T > 0$, PAG $\gamma_T$ satisfies $\gamma_T\pb{\rho_T(\cdot)} \preceq \gamma_{kT}\pb{\rho_{kT}(\cdot)}$.
\end{proposition}

Once $\gamma_T$ is found for $T \in [T_0, 2T_0]$ with some $T_0 > 0$, Proposition~\ref{pr: integer periods} can be used to lower-bound $\gamma_T$ for $T > 2T_0$.

\subsection{Typical Scenario}
\label{se: scenario}

We view the PAG concept as an \emph{add-on} to the classical \emph{input-to-state stability} (ISS) analysis rather than a proper refinement thereof. We have in mind a practical scenario of its use that shall be illustrated by an example in Section~\ref{se: example}. The scenario starts with validating the following two assumptions by any of the existing methods.

\begin{assumption}
    \label{as: LISS}
    Every solution $x(t)$ of~\eqref{eq: system general} with $u \in \calU$ and $x(0) \in \calX_0$ after some time enters a compact forward-invariant set $\calS(\norm{u}_\infty) \subset \bbR^n$.
\end{assumption}

Assumption~\ref{as: LISS} can be established, e.g., by showing that system~\eqref{eq: system general} is \emph{locally input-to-state stable} (LISS) \cite{sontagNewCharacterizationsInputtostate1996}. It implies the conservative AG
\begin{equation}
    \label{eq: scenario AG}
    \gamma(\norm{u}_\infty) = \max_{x \in \calS(\norm{u}_\infty)} \norm{Cx + g(x)}.
\end{equation}

\begin{assumption}
    \label{as: contraction}
    System~\eqref{eq: system general} with every $u \in \calU$ is \emph{uniformly exponentially contractive}~\cite{lohmillerContractionAnalysisNonlinear1998} in $\calS(\norm{u}_\infty)$, i.e., for every pair of its solutions $\bar x, \tilde x$ residing in $\calS(\norm{u}_\infty)$
    \begin{equation}
        \norm{\bar x(t) - \tilde x(t)} \leq c \ee^{-\sigma t} \norm{\bar x(0) - \tilde x(0)}
    \end{equation}
    for all $t \geq 0$ and some constants $c \geq 1$ and $\sigma > 0$ independent of $u$, $\bar x(0)$, and $\tilde x(0)$.
\end{assumption}

With Assumptions~\ref{as: LISS} and ~\ref{as: contraction} established, we proceed with the PAG analysis by assuming that $u \in \calU_T \subset \calU$. Then the stroboscopic map $x(0) \mapsto x(T)$ is a contraction that maps $\calS(\norm{\rho_T(u)}_1)$ into itself, and by Banach's theorem the map has a unique fixed point in $\calS(\norm{\rho_T(u)}_1)$. Thus, all solutions of~\eqref{eq: system general} with $u \in \calU_T$ and $x(0) \in \calX_0$ converge to a unique periodic solution $\hat x(t; u)$ contained in $\calS(\norm{\rho_T(u)}_1)$. Suppose that the AC/DC components of the output $\hat y$ produced on $\hat x$ can be estimated as
\begin{equation}
    \label{eq: y hat estimation scenario}
    \rho_T(\hat y) =
    \begin{bmatrix*}[l]
        \norm{\hat y\dc} \\
        \norm{\hat y\ac}_\infty
    \end{bmatrix*} \preceq \gamma_T\pb{\rho_T(u)}
\end{equation}
with some function $\gamma_T$. Then $\gamma_T$ is a PAG.

\subsection{Is PAG sharper than AG?}
\label{se: sharpness}

With the AG $\gamma$ from~\eqref{eq: scenario AG} using~\eqref{eq: 1-norm inf-norm inequality u} we could have immediately had for $u \in \calU_T$
\begin{equation}
    \norm{\hat y}_\infty \leq \gamma(\norm{u}_\infty) \leq \gamma\pb{\norm{\rho_T(u)}_1}
\end{equation}
due to the monotonicity of $\gamma$. On the other hand, with the PAG $\gamma_T$ we have from~\eqref{eq: y hat estimation scenario}
\begin{equation}
    \label{eq: 1-norm inf-norm inequality y}
    \norm{\hat y}_\infty \leq \norm{\hat y\dc} + \norm{\hat y\ac}_\infty
    \leq \normb{\gamma_T\pb{\rho_T(u)}}_1.
\end{equation}
One might say that PAG $\gamma_T$ is \emph{sharper} than AG~$\gamma$ if $\gamma_T$ is dominated by $\gamma$ in the sense of
\begin{equation}
    \label{eq: scenario PAG better than AG}
    \normb{\gamma_T\pb{\rho_T(u)}}_1 <
    \gamma\pb{\norm{\rho_T(u)}_1}.
\end{equation}
It is possible that, on the contrary, $\gamma$ turns out sharper than $\gamma_T$, perhaps in large part due to the crudeness of~\eqref{eq: 1-norm inf-norm inequality y}. Nevertheless, even in that case the \emph{separate} estimations of the output's AC/DC components provided by $\gamma_T$ might still be useful, e.g., if one of the components is ``small\rlap{.}''

\section{Linear Case}
\label{se: linear}

In the linear case, PAG can be found using the basic linear system theory.

\subsection{Calculating the PAG}

Consider system~\eqref{eq: system general} with $f = 0$ and $g = 0$. Let $G(s)$ be its transfer matrix and $H(t)$ its impulse response matrix:
\begin{alignat}{3}
    G(s) &:= C (sI - A)^{-1} B, \quad && s\in\bbC, \\
    H(t) &:= C \ee^{At} B, && t\geq 0.
\end{alignat}
Recall that $y(t) = \int_0^\infty H(\tau) u(t-\tau) \,\dif\tau$ and thus the exact AG is the linear function
\begin{equation}
    \label{eq: exact AG linear}
    \gamma(\norm{u}_\infty) = \int_0^\infty \norm{H(t)} \,\dif t \: \norm{u}_\infty.
\end{equation}
Let $H_T(t)$ be the $T$-periodic-impulse response:
\begin{equation}
    H_T(t) := \sum_{k=0}^{\infty} H(t + kT)
    \equiv C \ee^{At} (I - \ee^{AT})^{-1} B.
\end{equation}

\begin{theorem}
    \label{th: linear}
    \begin{subequations}
        The exact PAG of a stable linear system with transfer matrix $G$ and $T$-periodic-impulse response $H_T$ is
        \begin{equation}
            \gamma_T\pb{\rho_T(u)} = \Gamma_T \rho_T(u)
        \end{equation}
        where
        \begin{align}
            \Gamma_T &= \begin{bmatrix}
                \gamma\dc & 0 \\
                0 & \gamma\ac(T)
            \end{bmatrix}, \\
            \label{eq: dc gain linear}
            \gamma\dc &= \norm{G(0)}, \\
            \label{eq: ac gain linear exact}
            \gamma\ac(T) &= T \sup_{\substack{v\in\bbR^p \\ \norm{v} = 1}} D_{\median} (v \cdot H_T).
        \end{align}
    \end{subequations}
\end{theorem}

\begin{proof}
    Since $A$ is Hurwitz, the system's response to a $T$-periodic input $u(t) = u\dc + u\ac(t)$ globally converges to the $T$-periodic output $\hat y(t) = \hat y\dc + \hat y\ac(t)$ with
    \begin{subequations}
        \begin{align}
            \label{eq: dc linear transform}
            \hat y\dc &= G(0) u\dc,\\
            \label{eq: ac linear transform}
            \hat y\ac(t) &= \int_0^\infty H(\tau) u\ac(t-\tau) \,\dif\tau.
        \end{align}
    \end{subequations}
    From~\eqref{eq: dc linear transform} follows~\eqref{eq: dc gain linear}.
    
    Since $u\ac$ is $T$-periodic, \eqref{eq: ac linear transform} can be rearranged into the \emph{periodic convolution}
    \begin{equation}
        \label{eq: ac linear transform periodic}
        \hat y\ac(t) \equiv \int_0^T H_T(\tau) u\ac(t-\tau) \,\dif\tau.
    \end{equation}
    Thus,
    \begin{equation}
        \label{eq: output norm as supremum}
        \norm{\hat y\ac(t)} = \sup_{\substack{v\in\bbR^p \\ \norm{v} = 1}}
        \normx{\int_0^T v \cdot H_T(\tau) u\ac(t-\tau) \,\dif\tau}.
    \end{equation}
    If $\mean{u\ac}_T = 0$ then for every constant $\mu \in \bbR^p$
    \begin{equation}
        \norm{\hat y\ac(t)} \equiv \sup_{\substack{v\in\bbR^p \\ \norm{v} = 1}}
        \normx{\int_0^T \pb{v \cdot H_T(\tau) - \mu\trans} u\ac(t-\tau) \,\dif\tau}.
    \end{equation}
    The exact maximum of
    \begin{equation}
        \label{eq: ac integral to maximize}
        \normx{\int_0^T \pb{v \cdot H_T(\tau) - \mu\trans} u\ac(t-\tau) \,\dif\tau}
    \end{equation}
    with respect to $u\ac(\cdot)$ bounded by $\norm{u\ac}_\infty$ is delivered by
    \begin{equation}
        \label{eq: worst case input linear}
        u\ac(t-\tau) = \frac{v \cdot H_T(\tau) - \mu\trans}
        {\norm{v \cdot H_T(\tau) - \mu\trans}} \norm{u\ac}_\infty.
    \end{equation}
    In order to satisfy the requirement $\mean{u\ac}_T = 0$, it is necessary to take $\mu\trans = \median(v \cdot H_T)$ according to~\eqref{eq: median definition implicit}. With this $\mu$, the value of~\eqref{eq: ac integral to maximize} using~\eqref{eq: worst case input linear} is $T D_{\median}(v \cdot H_T) \norm{u\ac}_\infty$ which leads to~\eqref{eq: ac gain linear exact}.
\end{proof}

\begin{remark}
    A valid conservative value of $\gamma\ac$ is
    \begin{equation}
        \label{eq: ac gain linear simple}
        \gamma\ac(T) = \int_0^T \norm{H_T(t)} \,\dif t
    \end{equation}
    which is obtained by maximizing the norm of~\eqref{eq: ac linear transform periodic} ignoring the requirement $\mean{u\ac}_T = 0$. Incidentally, conservative AC-PAG~\eqref{eq: ac gain linear simple} is the same as the bound \cite[Eq.~(5.27)]{karafyllisRelationIOSgainsAsymptotic2021}.
\end{remark}

\begin{remark}
    The geometric median involved in~\eqref{eq: ac gain linear exact} is meant to be found numerically as the geometric median of finely sampled time series. Efficient numerical optimization methods exist for this purpose, e.g., see \cite{cohenGeometricMedianNearly2016}. For a SISO system, $\median H_T$ is the standard median (a value $\mu$ such that $H_T$ is half the time above $\mu$ and half the time below).
\end{remark}

\begin{figure}
    \centering
    \includegraphics[trim = 6.8cm 11cm 6.8cm 11.4cm, clip]{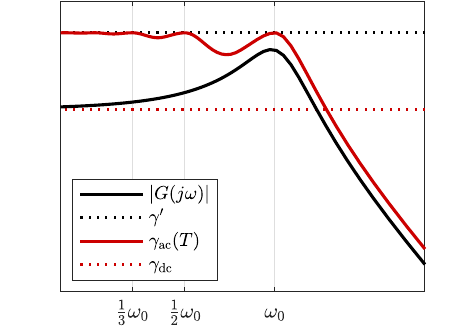}
    \caption{Linear input-output gains: frequency response $\abs{G(j\omega)}$, derivative $\gamma'$ of the classical asymptotic gain $\gamma$, and period-aware asymptotic gain $\gamma_T$ represented by its AC/DC components $\gamma\ac(T)$ and $\gamma\dc$ (Theorem~\ref{th: linear}). The horizontal axis is both the $\omega$-axis for $\abs{G(j\omega)}$ and $T$-axis for $\gamma\ac(T)$ with $T = 2\pi/\omega$. The system has a resonant frequency $\omega_0$.}
    \label{fig: linear}
\end{figure}

\subsection{Comparison of Three Gains}

For linear systems, we have the following asymptotic input-output gains:
\begin{enumerate}
    \item $\norm{G(j\omega)}$~-- relevant if the input is a pure sine wave with angular frequency $\omega$;
    \item PAG $\gamma_T$~-- relevant for bounded $T$-periodic inputs;
    \item AG $\gamma$~-- relevant for all uniformly bounded inputs.
\end{enumerate}
The set of $T$-periodic signals is a superset of sine waves with the corresponding \emph{fundamental} angular frequency $\omega = 2\pi/T$. Under this correspondence, each gain on the list above provides an output bound for a more general class of inputs than the previous one. The following relations therefore hold.

\begin{proposition}
    \label{pr: linear gains comparison}
    If $G$ is the transfer matrix of a stable linear system, $\gamma\ac$ and $\gamma\dc$ are the AC/DC components of the exact PAG from Theorem~\ref{th: linear}, and $\gamma'$ is the derivative of the exact AG $\gamma$ from~\eqref{eq: exact AG linear}, then for all $T > 0$
    \vspace{-1mm}
    \begin{subequations}
        \begin{alignat}{3}
            \norm{G(0)} &= \hspace{2.5mm} \gamma\dc &&\leq \gamma', \\
            \normb{G(j2\pi/T)} &\leq \gamma\ac(T) &&\leq \gamma'.
        \end{alignat}
    \end{subequations}
    \vspace{-5mm}
\end{proposition}

Figure~\ref{fig: linear} shows qualitatively the three gains for an example two-dimensional SISO LTI system. The AC component $\gamma\ac$ of the PAG captures the high-frequency attenuation feature of the system and converges to the classical AG for long-period signals. At the same time, the low-frequency gain is captured by the DC component $\gamma\dc$. In this sense, PAG occupies a ``middle ground'' between the classical AG and the frequency response.

We also note that, in agreement with intuition and Proposition~\ref{pr: integer periods}, the peaks of $\gamma\ac$ occur at the periods $2\pi k/\omega_0$ where $k \in \bbN$ and $\omega_0$ is the resonant frequency of the system.

Finally, observe that for the exact PAG
\begin{align}
    \normb{\gamma_T\pb{\rho_T(u)}}_1
    &\equiv \gamma\dc \norm{u\dc} + \gamma\ac(T) \norm{u\ac}_\infty \notag \\
    &\leq \max\set{\gamma\dc, \gamma\ac(T)} \norm{\rho_T(u)}_1
\end{align}
and by Proposition~\ref{pr: linear gains comparison} inequality~\eqref{eq: scenario PAG better than AG} holds, at least as a non-strict inequality. Therefore, in the linear case PAG is \emph{at least as sharp} as the classical AG (in the sense of Section~\ref{se: sharpness}).

\section{Nonlinear Case}
\label{se: nonlinear}

\subsection{General System}

Consider system~\eqref{eq: system general} and suppose that we are in the setting of Section~\ref{se: scenario}, i.e., Assumptions~\ref{as: LISS} and~\ref{as: contraction} hold and guarantee that for every $u \in \calU_T$ there exists a unique periodic solution $\hat x$ in a compact set $\calS(\norm{u}_\infty)$ that attracts all other solutions starting in $\calX_0$. Our goal is to attain a PAG by estimating $\rho_T(\hat y)$ of the periodic output $\hat y$. The estimation uses the following quadratic bounds on the nonlinear part of the system. They hold, e.g., if $f$ and $g$ have continuous second derivatives.

\begin{assumption}
    \label{as: quadratic bounds}
    For all arguments $x, \bar x \in \calS(u\mmax)$ and $\norm{u}, \norm{\bar u} \leq u\mmax$ functions $f$ and $g$ in~\eqref{eq: system general} satisfy
    \begin{subequations}
        \vspace{-2mm}
        \begin{multline}
            \label{eq: f quadratic bound}
            \normB{f(x,u) - f(\bar x, \bar u) \\
            - \pd[f]{x}(\bar x, \bar u) (x - \bar x)
            - \pd[f]{u}(\bar x, \bar u) (u - \bar u)} \\
            \leq M_f \pb{\norm{x - \bar x}^2 + \norm{u - \bar u}^2},
        \end{multline}
        \vspace{-4mm}
        \begin{equation}
            \label{eq: g quadratic bound}
            \normB{g(x) - g(\bar x)
            - \pd[g]{x}(\bar x) (x - \bar x)}
            \leq M_g \norm{x - \bar x}^2,
        \end{equation}
        \vspace{-5mm}
        \begin{gather}
            f(0,0) = 0, \quad \pd[f]{x}(0,0) = 0, \quad \pd[f]{u}(0, 0) = 0, \\
            g(0) = 0, \quad \pd[g]{x}(0) = 0
        \end{gather}
    \end{subequations}
    where $\partial f/\partial x$, $\partial f/\partial u$, and $\partial g/\partial x$ are the Jacobian matrices and $M_f, M_g > 0$ are some constants.
\end{assumption}


We start by estimating $f(\hat x, u)$ along the periodic solution $\hat x$ based on presupposed knowledge of $\rho_T(\hat x)$.

\begin{lemma}
    \label{le: bound on f hat}
    Let $\hat f(t) := f(\hat x(t), u(t))$ with any $T$-periodic $\hat x$ and $u$. Then under Assumption~\ref{as: quadratic bounds}
    \begin{equation}
        \label{eq: bound on f hat}
        \rho_T(\hat f) \preceq M_f
        \begin{bmatrix}
            \norm{\rho_T(\hat x)}^2 + \norm{\rho_T(u)}^2 \\
            2\pb{\norm{\hat x\ac}_\infty^2 + \norm{u\ac}_\infty^2}
        \end{bmatrix}.
    \end{equation}
\end{lemma}

\vspace{2mm}
\begin{proof}
    From~\eqref{eq: f quadratic bound} with $(\bar x, \bar u) = (0, 0)$ we obtain
    \begin{equation}
        \normb{f(\hat x\dc, u\dc)} \leq M_f\pb{\norm{\hat x\dc}^2 + \norm{u\dc}^2}
    \end{equation}
    and with $(\bar x, \bar u) = (\hat x\dc, u\dc)$~-- the ``Jensen gap'' estimation
    \begin{align}
        \label{eq: Jensen gap}
        \normb{\hat f\dc - f(\hat x\dc, u\dc)}
        &\equiv \normb{\meanb{f(\hat x, u)}_T - f\pb{\mean{\hat x}_T, \mean{u}_T}} \notag \\
        &\leq M_f\pb{\mean{\hat x\ac^2}_T + \mean{u\ac^2}_T} \notag \\
        &\leq M_f\pb{\norm{\hat x\ac}_\infty^2 + \norm{u\ac}_\infty^2}
    \end{align}
    which together produce the first component of the vector inequality~\eqref{eq: bound on f hat}. The second component follows from
    \begin{align}
        \norm{\hat f\ac}_\infty
        &\equiv \normb{f(\hat x, u) - \hat f\dc}_\infty \notag \\
        &\leq \normb{f(\hat x, u) - f(\hat x\dc, u\dc)}_\infty \!\!+ \normb{\hat f\dc - f(\hat x\dc, u\dc)} \notag \\
        &\leq 2M_f\pb{\norm{\hat x\ac}_\infty^2 + \norm{u\ac}_\infty^2}
    \end{align}
    with the latter estimation using~\eqref{eq: f quadratic bound} and~\eqref{eq: Jensen gap} again.
\end{proof}

Next we are going to bound $\rho_T(\hat x)$ and $\rho_T(\hat y)$ by treating the signal $\hat f = f(\hat x, u)$ as another input with presupposed $\rho_T(\hat f)$. The following linear PAGs obtainable via Theorem~\ref{th: linear} are used for that purpose:
\begin{subequations}
    \begin{itemize}
        \item $\gux_T(\rho) = \Gux_T \rho$ with $\Gux_T = \begin{bmatrix}
            \gux\dc & 0 \\ 0 & \gux\ac(T)
        \end{bmatrix}$ is the PAG of
        \begin{equation}
            \dot x = Ax + Bu, \quad
            y = x;
        \end{equation}
        \item $\gfx_T(\rho) = \Gfx_T \rho$ with $\Gfx_T  = \begin{bmatrix}
            \gfx\dc & 0 \\ 0 & \gfx\ac(T)
        \end{bmatrix}$ is the PAG of
        \begin{equation}
            \dot x = Ax + Fu, \quad
            y = x;
        \end{equation}
        \item $\guy_T(\rho) = \Guy_T \rho$ with $\Guy_T  = \begin{bmatrix}
            \guy\dc & 0 \\ 0 & \guy\ac(T)
        \end{bmatrix}$ is the PAG of
        \begin{equation}
            \dot x = Ax + Bu, \quad
            y = Cx;
        \end{equation}
        \item $\gfy_T(\rho) = \Gfy_T \rho$ with $\Gfy_T  = \begin{bmatrix}
            \gfy\dc & 0 \\ 0 & \gfy\ac(T)
        \end{bmatrix}$ is the PAG of
        \begin{equation}
            \dot x = Ax + Fu, \quad
            y = Cx.
        \end{equation}
    \end{itemize}
\end{subequations}
Then by the linearity and the triangle inequality
\begin{equation}
    \label{eq: bound on x hat}
    \rho_T(\hat x) \preceq \gux_T\pb{\rho_T(u)}
    + \gfx_T\pb{\rho_T(\hat f)}
\end{equation}
and
\begin{multline}
    \label{eq: bound on y hat}
    \rho_T(\hat y) \preceq \guy_T\pb{\rho_T(u)}
    + \gfy_T\pb{\rho_T(\hat f)} \\
    + M_g \begin{bmatrix}
        \norm{\rho_T(\hat x)}^2 \\
        2\norm{\hat x\ac}_\infty^2
    \end{bmatrix}
\end{multline}
where the last summand comes from the estimation of $\rho_T\pb{g(\hat x)}$ similarly to Lemma~\ref{le: bound on f hat}.

The following theorem resolves the ``circular'' estimations~\eqref{eq: bound on f hat}, \eqref{eq: bound on x hat}, and~\eqref{eq: bound on y hat} taking into account the bound
\begin{equation}
    \label{eq: a priori bound x hat}
    \norm{\hat x}_\infty \leq \max_{x \in \calS(\norm{\rho_T(u)}_1)} \norm{x}
\end{equation}
established via LISS (Assumption~\ref{as: LISS}) and~\eqref{eq: 1-norm inf-norm inequality u}.

\begin{theorem}
    \label{th: nonlinear PAG}
    Under Assumptions~\ref{as: LISS}--\ref{as: quadratic bounds}, given $u \in \calU_T$, let
    \begin{subequations}
        \label{eq: solution xi ac}
        \begin{equation}
            \xi\ac = \begin{cases}
                \dfrac{1 - \sqrt{1 - 4 a\ac c\ac}}{2a\ac},
                & 4a\ac b^2 + c\ac \leq 2b, \\
                2b, & \mathrm{otherwise}
            \end{cases}
        \end{equation}
        where
        \begin{align}
            b &= \max\nolimits_{x \in \calS(\norm{\rho_T(u)}_1)} \norm{x}, \\
            a\ac &= 2M_f \gfx\ac(T), \\
            c\ac &= 2M_f \gfx\ac(T) \norm{u\ac}_\infty^2 + \gux\ac(T) \norm{u\ac}_\infty.
        \end{align}
    \end{subequations}
    After that, let
    \begin{subequations}
        \label{eq: solution xi dc}
        \begin{equation}
            \xi\dc = \begin{cases}
                \dfrac{1 - \sqrt{1 - 4 a\dc c\dc}}{2a\dc},
                & a\dc b^2 + c\dc \leq b, \\
                b, & \mathrm{otherwise}
            \end{cases}
        \end{equation}
        where
        \begin{align}
            a\dc &= M_f \gfx\dc, \\
            c\dc &= M_f \gfx\dc \pb{\xi\ac^2 + \norm{\rho_T(u)}^2} + \gux\dc \norm{u\dc}.
        \end{align}
    \end{subequations}
    Then
    \begin{align}
        \gamma_T(\rho_T(u)) &= \Guy_T \rho_T(u)
        + M_f \Gfy_T \begin{bmatrix}
            \norm{\rho_T(u)}^2 \\ 2\norm{u\ac}_\infty^2
        \end{bmatrix} \notag \\
        \label{eq: nonlinear PAG}
        &\phantom{{}=} + \pb{M_f \Gfy_T + M_g \bracksx{\begin{smallmatrix}
            1 & 0 \\ 0 & 1
        \end{smallmatrix}}} \begin{bmatrix}
            \xi\dc^2 + \xi\ac^2 \\ 2\xi\ac^2
        \end{bmatrix}
    \end{align}
    is a valid conservative PAG of~\eqref{eq: system general}.
\end{theorem}

\begin{proof}
    Combining the inequalities~\eqref{eq: bound on f hat} and~\eqref{eq: bound on x hat} and writing them out element-wise we arrive at
    \begin{subequations}
        \begin{alignat}{4}
            a\dc &\norm{\hat x\dc}^2 &&- \norm{\hat x\dc} &{}+ c\dc &\geq 0, \\
            a\ac &\norm{\hat x\ac}_\infty^2 &&- \norm{\hat x\ac}_\infty &{}+ c\ac &\geq 0,
        \end{alignat}
    \end{subequations}
    the coefficients $a$ and $c$ being the same as in~\eqref{eq: solution xi ac} and~\eqref{eq: solution xi dc} with the only exception that $c\dc$ at this point has $\norm{\hat x\ac}_\infty$ in it in place of $\xi\ac$.
    From~\eqref{eq: a priori bound x hat} we have
    \begin{equation}
        \norm{\hat x\dc} \leq b, \quad
        \norm{\hat x\ac}_\infty \leq 2b.
    \end{equation}
    All of this implies
    \begin{subequations}
        \label{eq: set for x dc and x ac}
        \begin{alignat}{3}
            \label{eq: set for x dc}
            &\norm{\hat x\dc} &&\leq \max\set{\xi \in [0,b] \colon
            &&a\dc \xi^2 - \xi + c\dc \geq 0}, \\
            \label{eq: set for x ac}
            &\norm{\hat x\ac}_\infty &&\leq \max\set{\xi \in [0,2b] \colon
            &&a\ac \xi^2 - \xi + c\ac \geq 0}.
        \end{alignat}
    \end{subequations}
    Estimation~\eqref{eq: set for x ac} is independent of~\eqref{eq: set for x dc} and can be simplified first. It yields $\norm{\hat x\ac}_\infty \leq \xi\ac$. Thus obtained $\xi\ac$ is substituted in place of $\norm{\hat x\ac}_\infty$ that appears in the coefficient $c\dc$ in~\eqref{eq: set for x dc} which then implies $\norm{\hat x\dc} \leq \xi\dc$ and overall
    \begin{equation}
        \label{eq: bound on x hat solved}
        \rho_T(\hat x) \preceq
        \begin{bmatrix}
            \xi\dc \\ \xi\ac
        \end{bmatrix}.
    \end{equation}
    Gathering~\eqref{eq: bound on f hat}, \eqref{eq: bound on y hat}, and \eqref{eq: bound on x hat solved} we attain~\eqref{eq: nonlinear PAG}.
    
    Note that $\xi_{\rm dc, ac}$ are determined as the right-hand sides of~\eqref{eq: set for x dc and x ac} and thus are non-decreasing with $a_{\rm dc, ac}$, $c_{\rm dc, ac}$, and $b$. The latter coefficients are non-decreasing with $\norm{u\dc}$ and $\norm{u\ac}_\infty$. Therefore, \eqref{eq: nonlinear PAG} is indeed $\preceq$-non-decreasing as required by Definition~\ref{def: pag}.
\end{proof}

\begin{remark}
    Unlike the linear PAG~\eqref{eq: exact AG linear} whose matrix is diagonal, nonlinear PAG~\eqref{eq: nonlinear PAG} ``mixes'' the AC and DC components: AC ``spills'' into DC but not the other way around. Such mixing is to be expected from a nonlinear system because every even power of a pure sine wave has non-zero DC offset. As nonlinearity vanishes with $M_f, M_g \to 0$, conservative PAG~\eqref{eq: nonlinear PAG} converges to the exact linear PAG~\eqref{eq: exact AG linear}.
\end{remark}

\subsection{System of a Special Structure}

In the example of Section~\ref{se: example} our system is of a special ``control-Lurie'' type~-- with a linear output map and nonlinearity $f$ dependent only on $y$ rather than the entire $x$:
\begin{subequations}
    \label{eq: system special}
    \begin{align}
        \dot x &= Ax + Bu + F f(y,u), \\
        y &= Cx.
    \end{align}
\end{subequations}
For~\eqref{eq: system special}, Assumption~\ref{as: quadratic bounds} is replaced with the following.

\begin{assumption}
    \label{as: quadratic bounds special}
    For all arguments $x, \bar x \in \calS(u\mmax)$ and $\norm{u}, \norm{\bar u} \leq u\mmax$ function $f$ in~\eqref{eq: system special} satisfies
    \begin{subequations}
        \begin{multline}
            \label{eq: f quadratic bound special}
            \normB{f(Cx, u) - f(C \bar x, \bar u) \\
            - \pd[f]{y}(C \bar x, \bar u) C (x - \bar x)
            - \pd[f]{u}(C \bar x, \bar u) (u - \bar u)} \\
            \leq M_f \pb{\norm{C(x - \bar x)}^2 + \norm{u - \bar u}^2},
        \end{multline}
        \begin{equation}
            f(0,0) = 0, \quad \pd[f]{y}(0,0) = 0, \quad \pd[f]{u}(0,0) = 0
        \end{equation}
    \end{subequations}
    where $\partial f/\partial y$ and $\partial f/\partial u$ are the Jacobian matrices and $M_f > 0$ is some constant.
\end{assumption}

Theorem~\ref{th: nonlinear PAG} specializes to~\eqref{eq: system special} in the following way.

\begin{corollary}
    \label{cor: special case}
    Under Assumptions~\ref{as: LISS}, \ref{as: contraction}, \ref{as: quadratic bounds special}, given $u \in \calU_T$, let
    \begin{subequations}
        \label{eq: solution xi ac special}
        \begin{equation}
            \eta\ac = \begin{cases}
                \dfrac{1 - \sqrt{1 - 4 a\ac c\ac}}{2a\ac},
                & 4a\ac b^2 + c\ac \leq 2b, \\
                2b, & \mathrm{otherwise}
            \end{cases}
        \end{equation}
        where
        \begin{align}
            b &= \gamma(\norm{\rho_T(u)}_1) \overset{\eqref{eq: scenario AG}}{\equiv}
            \max\nolimits_{x \in \calS(\norm{\rho_T(u)}_1)} \norm{Cx}, \\
            a\ac &= 2M_f \gfy\ac(T), \\
            c\ac &= 2M_f \gfy\ac(T) \norm{u\ac}_\infty^2 + \guy\ac(T) \norm{u\ac}_\infty.
        \end{align}
    \end{subequations}
    After that, let
    \begin{subequations}
        \label{eq: solution xi dc special}
        \begin{equation}
            \eta\dc = \begin{cases}
                \dfrac{1 - \sqrt{1 - 4 a\dc c\dc}}{2a\dc},
                & a\dc b^2 + c\dc \leq b, \\
                b, & \mathrm{otherwise}
            \end{cases}
        \end{equation}
        where
        \begin{align}
            a\dc &= M_f \gfy\dc, \\
            c\dc &= M_f \gfy\dc \pb{\eta\ac^2 + \norm{\rho_T(u)}^2} + \guy\dc \norm{u\dc}.
        \end{align}
    \end{subequations}
    Then
    \begin{equation}
        \label{eq: nonlinear PAG special}
        \gamma_T(\rho_T(u)) = \begin{bmatrix}
            \eta\dc \\ \eta\ac
        \end{bmatrix}
    \end{equation}
    is a valid conservative PAG of~\eqref{eq: system special}.
\end{corollary}

\begin{proof}
    Full-state estimation~\eqref{eq: bound on x hat} is not needed in this case. Instead of~\eqref{eq: bound on f hat} and~\eqref{eq: bound on y hat} we have
    \begin{subequations}
        \label{eq: special case estimations f hat y hat}
        \begin{align}
            \rho_T(\hat f) &\preceq M_f
            \begin{bmatrix}
                \norm{\rho_T(\hat y)}^2 + \norm{\rho_T(u)}^2 \\
                2\pb{\norm{\hat y\ac}_\infty^2 + \norm{u\ac}_\infty^2}
            \end{bmatrix}, \\
            \rho_T(\hat y) &\preceq \guy_T(\rho_T(u))
            + \gfy_T(\rho_T(\hat f)).
        \end{align}
    \end{subequations}
    Applying the same logic as in the proof of Theorem~\ref{th: nonlinear PAG}, inequalities~\eqref{eq: special case estimations f hat y hat} in view of the AG bound
    \begin{equation}
        \rho_T(\hat y) \preceq \begin{bmatrix}
            b \\ 2b
        \end{bmatrix}
    \end{equation}
    can be transformed into
    \begin{equation}
        \rho_T(\hat y) \preceq \begin{bmatrix}
            \eta\dc \\ \eta\ac
        \end{bmatrix}
    \end{equation}
    confirming PAG~\eqref{eq: nonlinear PAG special}.
\end{proof}

\section{Example}
\label{se: example}

\subsection{The Phase-Locked Loop System}

Consider a grid-following voltage source converter connected to the power grid. Synchronization of the converter with the grid is often achieved by a circuit called the \emph{phase-locked loop} (PLL). Its \emph{input} is the grid voltage sensed by the converter. The \emph{error dynamics} of a common PLL design (the synchronous reference frame PLL) are~\cite{escobarNonlinearStabilityAnalysis2021}
\begin{subequations}
    \label{eq: system example}
    \begin{align}
        \dot\theta &= k_p v_q(\theta, v) + \omega, \\
        \dot\omega &= k_i v_q(\theta, v)
    \end{align}
\end{subequations}
where
$
    v_q(\theta, v) = -\sin\theta + v \cdot \begin{bmatrix}
        \cos\theta \\ \sin\theta
    \end{bmatrix}
$,
$\theta$ and $\omega$ are the estimation errors of the grid's phase angle and frequency, $k_p$ and $k_i$ are parameters, and $v \in \bbR^2$ is the input (perturbation of the voltage measured by the converter).

The \emph{output} of the converter is the AC current injected into the grid. The current depends on $\theta$ but for the sake of this example we disregard the dependence and consider $\theta$ the output. Taking
$
    x = \begin{bmatrix}
        \theta & \omega
    \end{bmatrix}\trans
$,
$u = v$, and $y = \theta$ we obtain the PLL error dynamics as system~\eqref{eq: system special} with
\begin{subequations}
    \begin{gather}
        A = \begin{bmatrix}
            -k_p & 1 \\
            -k_i & 0
        \end{bmatrix}, \quad
        B = \begin{bmatrix}
            k_p & 0 \\
            k_i & 0
        \end{bmatrix}, \quad
        F = \begin{bmatrix}
            k_p \\ k_i
        \end{bmatrix}, \\
        C = \begin{bmatrix}
            1 & 0
        \end{bmatrix}\!,\:
        f(y,u) = y - \sin y + u \cdot \begin{bmatrix}
            \cos y - 1 \\
            \sin y
        \end{bmatrix}.
    \end{gather}
\end{subequations}

Parameters $k_p$ and $k_i$ are commonly presented in the form
\begin{equation}
    k_p = 2\zeta\omega_c, \quad k_i = \omega_c^2
\end{equation}
where $\omega_c$ (rad/s) is the \emph{bandwidth} of the PLL and $\zeta$ is the dimensionless \emph{damping factor}. We make the following tuning decisions:
\begin{itemize}
    \item $\zeta = 1/\sqrt{2}$~-- Wiener-optimal choice that balances tracking speed and filtering performance~\cite{chungPhaseTrackingSystem2000};
    \item $\omega_c = 2\pi 10\ \mathrm{rad/s}$~-- well below the main grid frequency (let us assume that the latter is 50~Hz); this is known as \emph{low-gain} tuning and has been recommended for distorted grid conditions \cite{freijedoTuningPhaselockedLoops2009}.
\end{itemize}

\subsection{Classical AG and Contraction}
\label{se: example contraction}

As explained in Section~\ref{se: scenario}, we first have to satisfy Assumptions~\ref{as: LISS} and~\ref{as: contraction}.

Assumption~\ref{as: LISS} can be validated by the method of \emph{two-dimensional comparison systems}~\cite{ponomarevNonlinearAnalysisSynchronous2024}. Given a uniform input bound $\norm{u}_\infty$, the method yields a forward-invariant set $\calS(\norm{u}_\infty)$ in the $x$-plane and an estimation of its domain of attraction, see~\cite[Fig.~4]{ponomarevNonlinearAnalysisSynchronous2024}. The asymptotic gain $\gamma(\norm{u}_\infty)$ can then be found via~\eqref{eq: scenario AG}. A powerful feature of the comparison method is that the border of $\calS(\norm{u}_\infty)$ is an actual trajectory of the system under some bang-bang input. Therefore, the AG found in this way is \emph{exact}.

Assumption~\ref{as: contraction} (contraction) can be checked by linearizing the system about a solution that lies in $\calS(\norm{u}_\infty)$. The linearization is generally time-varying but in our case has particular structure that enables a simple uniform exponential stability test~\cite[Lemma~2]{ponomarevNonlinearAnalysisSynchronous2024}.

The two assumptions limit the values of $\norm{u}_\infty$ that we can consider in the following.

\begin{figure}
    \centering
    \hspace{8mm}
    \textbf{Pure AC Input}
    \includegraphics[trim = 6.5cm 11cm 7cm 10.7cm, clip]{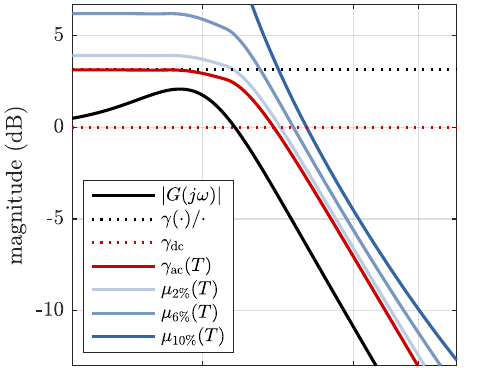}

    \hspace{8mm}
    \textbf{Input with AC and DC Components}
    \includegraphics[trim = 6.5cm 11cm 7cm 10.7cm, clip]{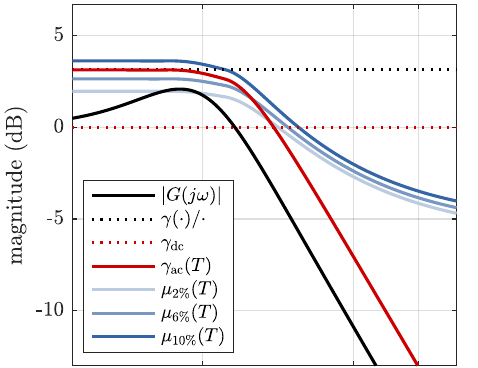}
    
    \hspace{8mm}
    \textbf{Pure DC Input}
    \includegraphics[trim = 6.5cm 10.5cm 7cm 10.7cm, clip]{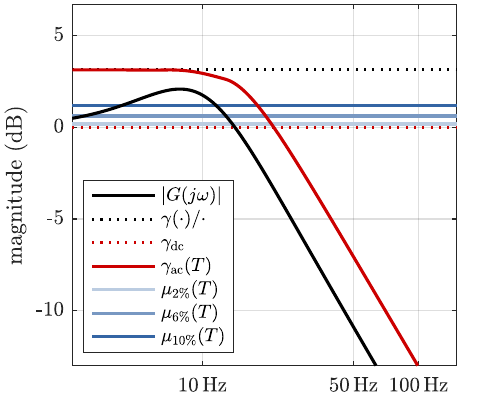}
    \caption{Input-output gains of the example system under different magnitudes and AC/DC compositions of the input: $\abs{G(j\omega)}$~-- frequency response of the linearized system; $\gamma(\cdot)/\cdot$~-- average derivative of the exact nonlinear AG $\gamma$; $\gamma\dc$ and $\gamma\ac(T)$~-- components of the exact PAG of the linearized system (Theorem~\ref{th: linear}); $\mu_\ell(T)$~-- quantities~\eqref{eq: example output measure}. The horizontal axis is both the $T$-axis for $\gamma\ac(T)$ and $\mu(T)$ and $\omega$-axis for $\abs{G(j\omega)}$ with $T = 2\pi/\omega$.}
    \label{fig: pll}
\end{figure}

\subsection{PAG Analysis}

Application of Theorem~\ref{th: linear} and Corollary~\ref{cor: special case} is straightforward and leads to the exact PAG of the linearized system, represented by its separate DC and AC components $\gamma\dc$ and $\gamma\ac(T)$, and a conservative nonlinear PAG $\gamma_T(\rho_T(u))$.

In order to facilitate the comparison of linear and nonlinear gains, we impose the following restrictions. \emph{Firstly}, we consider only three levels $\ell$ of the input magnitude:
\begin{equation}
    \norm{u}_\infty \leq \ell, \quad \ell \in \set{2\%, 6\%, 10\%}.
\end{equation}
Perturbation $u$ is measured in parts of the nominal grid voltage; $10\%$ is already a large magnitude that would rarely occur in practice. Following Section~\ref{se: example contraction}, at each level $\ell$ we find a forward-invariant set and establish contractivity inside it. At $\ell = 10\%$ the set is shaped like the set $\calE$ in~\cite[Fig.~4a]{ponomarevNonlinearAnalysisSynchronous2024} and has half-diameter of about $0.14$~rad in the $x_1$ direction~-- small enough to ensure contractivity by linearization.

\emph{Secondly}, we restrict the input's AC/DC composition to the following three options:
\begin{enumerate}
    \item Pure AC inputs: $u\dc = 0$ and $\norm{u\ac}_\infty \leq \ell$.
    \item Inputs with equally bounded AC/DC components: $\norm{u\dc} \leq \ell/2$ and $\norm{u\ac}_\infty \leq \ell/2$.
    \item Pure DC inputs: $\norm{u\dc} \leq \ell$ and $u\ac = 0$.
\end{enumerate}

\emph{Finally}, instead of the \emph{two-dimensional nonlinear} PAG $\gamma_T$ we plot the \emph{scalar} $\mu_\ell(T)$ which may be called the ``average slope of $\norm{\gamma_T}_1$ over $\ell$-bounded inputs\rlap{.}'' For the above input compositions, $\mu_\ell$ is defined as follows:
\begin{subequations}
    \label{eq: example output measure}
    \begin{enumerate}
        \item For pure AC inputs:
        \begin{equation}
            \mu_\ell(T) = \frac1\ell \normx{\gamma_T\px{\begin{bmatrix}
                0 \\ \ell
            \end{bmatrix}}}_1.
        \end{equation}
        \item For inputs with equally bounded AC/DC components:
        \begin{equation}
            \mu_\ell(T) = \frac1\ell \normx{\gamma_T\px{\begin{bmatrix}
                \ell/2 \\ \ell/2
            \end{bmatrix}}}_1.
        \end{equation}
        \item For pure DC inputs:
        \begin{equation}
            \mu_\ell(T) = \frac1\ell \normx{\gamma_T\px{\begin{bmatrix}
                \ell \\ 0
            \end{bmatrix}}}_1.
        \end{equation}
    \end{enumerate}
\end{subequations}

Figure~\ref{fig: pll} illustrates the numerical results:
\begin{itemize}
    \item frequency response of the linearized system;
    \item average derivative $\gamma(\cdot)/\cdot$ of the comparison-based exact AG $\gamma$ (although $\gamma$ is nonlinear, its derivative is almost constant for inputs bounded by $10\%$);
    \item exact PAG of the linearized system from Theorem~\ref{th: linear};
    \item ``average slopes''~\eqref{eq: example output measure} of the conservative PAG $\gamma_T$ from Corollary~\ref{cor: special case}.
\end{itemize}
The plots confirm and quantify the ``high-frequency attenuation'' behavior for AC-dominant periodic inputs high above the PLL bandwidth (10\,Hz)~-- e.g., harmonics of the main grid frequency (signals with periods 0.02\,s, 0.01\,s, etc.).

Voltage perturbation is often dominated by harmonics of the main grid frequency~\cite{erogluHarmonicProblemsRenewable2021}. Accordingly, let us next assume that the input $u$ is $T$-periodic with $T = 0.02$\,s and compare the PAG estimations to the actual oscillations that may occur in the system. Figure~\ref{fig: pll sim} presents the waveforms generated in response to randomized inputs (sums of harmonics with random amplitudes and phase shifts) in the several aforementioned cases regarding the input magnitude and AC/DC composition. Additionally, the bang-bang input~\eqref{eq: worst case input linear} has been applied which produced the triangular-looking waves (bang-bang is the worst-case input for the linearized system).

Comparing the randomized simulations, our PAG, and classical AG, we conclude that the PAG is indeed tighter than AG, particularly for pure AC inputs, and often appears to be close to the exact output bounds. Since our AG is exact, all improvement should be attributed to the fact that PAG makes use of the input periodicity.

\section{Conclusions}

We introduce the period-aware asymptotic gain (PAG)~-- an input-output gain for nonlinear systems with arbitrary bounded periodic input. Similarly to the classical asymptotic gain, it relates supremum norms of the input and output. However, it also takes into account the input periodicity, resembling a linear system's frequency response. The gain allows one to speak of bandwidth, low-pass behavior, etc. for nonlinear systems in a quantifiable manner. Future analysis will include inputs that are sums of signals with commensurable or incommensurable periods and almost periodic signals. Application of the PAG to the analysis of periodically forced networks is of interest as well.

\begin{figure}
    \centering
    \hspace{8mm}
    \textbf{Pure AC Input}
    \includegraphics[trim = 6.5cm 10.8cm 7cm 10.5cm, clip]{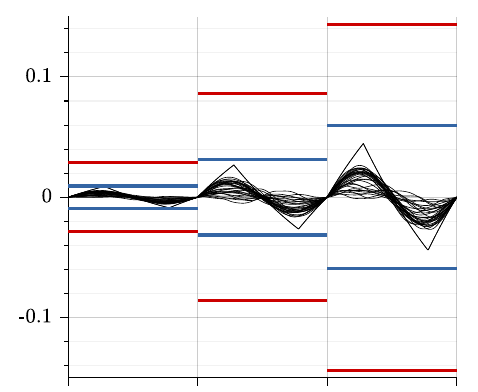}

    \hspace{8mm}
    \textbf{Input with AC and DC Components}
    \includegraphics[trim = 6.5cm 10.8cm 7cm 10.5cm, clip]{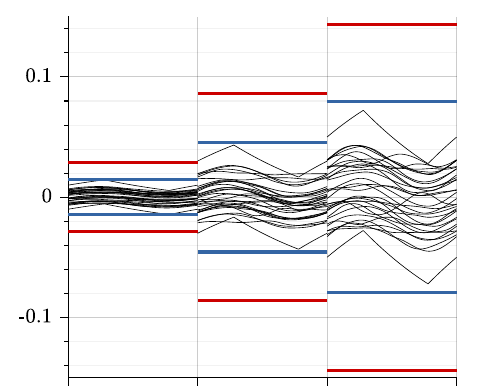}

    \hspace{8mm}
    \textbf{Pure DC Input}
    \includegraphics[trim = 6.5cm 10.8cm 7cm 10.5cm, clip]{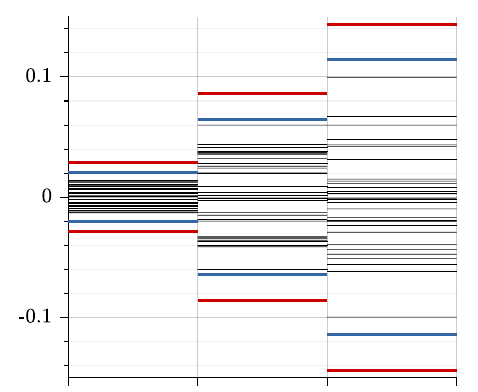}

    \vspace{-2mm}
    \hspace{6.6mm}
    $\underbrace{\hspace{21mm}}_{\norm{u}_\infty \leq 2\%}$
    $\underbrace{\hspace{21mm}}_{\norm{u}_\infty \leq 6\%}$
    $\underbrace{\hspace{21mm}}_{\norm{u}_\infty \leq 10\%}$
    \caption{Thin lines: output waveforms in response to randomized $T$-periodic inputs with $T = 0.02$\,s and magnitudes (left to right) $\norm{u}_\infty \leq 2\%$, $\norm{u}_\infty \leq 6\%$, and $\norm{u}_\infty \leq 10\%$. Thick lines: amplitude estimations by the exact classical AG (red) and conservative proposed PAG (blue).}
    \label{fig: pll sim}
\end{figure}

\bibliographystyle{IEEEtran}
\bibliography{paper}

\end{document}